\documentclass[11pt,reqno]{amsart}
  \usepackage{hyperref}
  \usepackage{amsmath,amssymb}
  \usepackage{xcolor}
  \usepackage{amsthm}
  \usepackage{centernot}
  \theoremstyle{plain}
 \usepackage[pagewise]{lineno}
  \newtheorem{thm}{Theorem}[section]
  \newtheorem{lemma}[thm]{Lemma}
  \newtheorem{cor}[thm]{Corollary}

  \theoremstyle{definition}
  \newtheorem{defn}[thm]{Definition}
  \newtheorem{example}[thm]{Example}

  \newtheorem{rem}[thm]{Remark}
  
  \numberwithin{equation}{section}

  \newcommand{\interior}[1]{{\kern0pt#1}^{\mathrm{o}}}

  \makeatletter
  \@namedef{subjclassname@2020}{%
  	\textup{2020} Mathematics Subject Classification}
  \makeatother

\begin{document}
  	

  	\title{Induced Homeomorphism and Atsuji Hyperspaces}
  	
  	\author[A. Gupta]{A. K. Gupta$^{\dagger *}$}
  	
  	\address{$^\dagger$Department of Mathematics\\ National Institute of Technology Meghalaya\\ Shillong 793003\\ Meghalaya\\ India}
  	\email{ajitkumar.gupta@nitm.ac.in}
  	\author[S. Mukherjee]{S. Mukherjee{$^{\dagger}$}}
  	\address{$^\dagger$Department of Mathematics\\ National Institute of Technology Meghalaya\\ Shillong 793003\\Meghalaya\\ India}
  	\email{saikat.mukherjee@nitm.ac.in}

  	$\thanks{*Corresponding author}$
  	

  	\begin{abstract} 
  	   Given uniformly homeomorphic metric spaces $X$ and $Y$, it is proved that the hyperspaces $C(X)$ and $C(Y)$ are uniformly homeomorphic, where $C(X)$ denotes the collection of all nonempty closed subsets of $X$, and is endowed with Hausdorff distance. Gerald Beer has proved that the hyperspace $C(X)$ is Atsuji when $X$ is either compact or uniformly discrete. An Atsuji space is a generalization of compact metric spaces as well as of uniformly discrete spaces. In this article, we investigate the space $C(X)$ when $X$ is Atsuji, and a class of Atsuji subspaces of $C(X)$ is obtained. Using the obtained results, some fixed point results for continuous maps on Atsuji spaces are obtained.
  	     	     	    
  	\end{abstract}

  		\maketitle

  	\section{Introduction}


Let $(W,\tau)$ be a topological space in which every singleton set $\{x\}$ is closed. A \textit{hyperspace} of the space $W$ is the collection $C(W)$ of nonempty closed subsets of $W$, endowed with a topology $\tau'$ such that the mapping $I:(W,\tau) \to (C(W),\tau')$, defined as $I(x)=\{x\}$, is a homeomorphism onto its range. This suitable topology $\tau'$ is called \textit{hypertopology} or \textit{hyperspace topology} \cite{rl94}. 
If $W$ is a metric space, then the induced Hausdorff distance $H$, an extended real-valued metric on $C(W)$, gives a hypertopology.

	A metric space $X$ is said to be \textit{Atsuji space} if each continuous map from $X$ to a metric space $Y$ is uniformly continuous \cite{gb86}.
	 The property of being Atsuji lies in between the completeness and the compactness. For last few decades, the theory of these spaces has attracted the attention of several researchers. These spaces demonstrate not only several interesting internal characterizations (see \cite{ma16, gb86, tj07}) but also exhibit several interesting external characterizations in the theory of hyperspaces (see \cite{gb87, gb92, tj08}). Atsuji spaces are also known as normal metric spaces (see \cite{sg65}) and as Lebesgue metric spaces (see \cite{sb81, sr93}).

	  In \cite{gb85}, Gerald Beer has shown that a metric space $X$ is either compact or uniformly discrete if and only if its Hausdorff hyperspace $C(X)$ is Atsuji. A compact metric space and a uniformly discrete space both are Atsuji spaces. In this manuscript, we take $X$ to be Atsuji and investigate the space $C(X)$. It is found that the space $C(X)$ fails to be Atsuji, however $C(X)$ contains a class of Atsuji subspaces. These Atsuji subspaces include the completions of those point-finite collections in $C(X)$ which contain all the singletons $\{x\}, x\in X$.
  	  	
 This article is organized as follows. Section \ref{Pre} contains some preliminaries required for the discussion in later sections.  In Section \ref{Ind_Map}, it is shown that if $X$ and $Y$ are uniformly homeomorphic, then their corresponding Hausdorff hyperspaces $C(X)$ and $ C(Y)$ are so. Section \ref{AtsjHprspcs} presents the sufficient conditions for the subspaces of $C(X)$ to be Atsuji.  In this regard, an open problem for the existence of maximal Atsuji subspace of $C(X)$ is also placed. Applying the obtained results, in Section \ref{FxdPntsRslts}, we acquire fixed point results for continuous maps on Atsuji spaces.

  	\section{Preliminaries}\label{Pre}
Given a subset $A$ in a metric space $X$, we denote the set of all limit points of $A$ by $A'$, and the complement of $A$ in $X$ by $A^c~ (\text{or }X\setminus A)$. An open ball in $X$, centered at $x\in X$ with radius $\epsilon>0$ is denoted by $B(x,\epsilon)$. The set $\bigcup\limits_{x\in A}B(x,\epsilon)$ is called the $\epsilon$-neighborhood of $A$ and is denoted by $N_{\epsilon}(A)$. 
  	
  	\begin{defn} The {\it Hausdorff distance}, $H$, of two nonempty subsets $A,B $ of a metric space $(X,d)$ is defined as $H(A,B)=\max\{\sup\limits_{x\in A}d(x,B),$ $\sup\limits_{x\in B}d(x,A)\}$, where $d(x,A)=\inf\limits_{y\in A}d(x,y)$.
  	\end{defn}

\begin{defn}
	A metric space $X$ is said to be an \textit{Atsuji space} if the set of limit points $X'$, and for each $\epsilon>0$, the set $[N_\epsilon (X')]^c$ is uniformly discrete.
\end{defn}


\begin{thm}\label{AtsujiI(x)}\cite{tj08}
	Let $(X,d)$ be a metric space. The completion $(\hat{X},d)$ is an Atsuji space if and only if every sequence $\{x_n\}$ in $X$ with $\lim\limits_{n\to\infty} I (x_n)=0$ has a Cauchy subsequence, where $I(x)=d(x,X\setminus \{x\}), x\in X$.
\end{thm}


A topological vector space $X$ is said to be \textit{locally convex} if there is a local base $\mathcal B$ at $\textbf{0}$ (the zero vector) whose members are convex. 

\begin{thm} [Tychonoff's Fixed Point Theorem] \cite{sa75}
    Let $A$ be a compact convex subset of a locally convex topological vector space. If $f:A\to A$ is a continuous map, then $f$ has a fixed point.
\end{thm}

 \section{Induced Map on Hyperspaces}\label{Ind_Map}
  	
  	Consider a map $f:X\to Y$. Let $P(X)$ denote the collection of all nonempty subsets of $X$. Then, the \textit{induced map} $F : P(X)\to P(Y)$ is defined as $A\mapsto\{f(x):x\in A\}$,  $A\in P(X)$.  In this article, for convenience, the induced map, generated by a given map $f$, will be denoted by the corresponding capital letter $F$.
  	

The induced map plays some roles in fixed point theory. In this article we use this map to find a fixed point for its inducing map. The induced map has been used by Nadler in \cite{sb87}, to discuss the fixed point property of some hyperspaces of certain continua. 
  	 
 For metric spaces $(X,d)$, $(Y,d')$, we denote their corresponding Hausdorff hyperspaces of nonempty closed subsets by  $(C(X),H)$, $(C(Y),$ $H')$, respectively; where $H$ and $H'$ are Hausdorff distances induced by the metrics $d$ and $d'$, respectively. 


 A \textit{uniform homeomorphism} $f$ from a metric space $X$ to a metric space $Y$ is a bijective map such that $f$ and $f^{-1}$ are uniformly continuous.
   
The following result provides a sufficient condition for the hyperspace $C(X)$ of a metric space $X$ to be uniformly homeomorphic to a hyperspace $C(Y)$, as well as is used to derive a fixed point result in Section \ref{FxdPntsRslts}.

\begin{thm}\label{fHomFhom}
	Let $X,Y$ be metric spaces. Then, $f$ is a uniform homeomorphism from $X$ to $Y$ if and only if the induced map $F$, defined on the hyperspace $C(X)$, is a uniform homeomorphism from $C(X)$ to the hyperspace $C(Y)$.
\end{thm}
	\begin{proof}

		By continuity of $f^{-1}$, we have $F(A)$ is closed in $Y$ for each $A\in C(X)$.
		
		We show that $F$ is uniformly continuous. By uniform continuity of $f$, for each $\epsilon>0$ there is a $\delta>0$ such that $d'(f(x),f(y))<\epsilon$ whenever $d(x,y)<\delta$ for all $x,y\in X$. Consider $H(A,B)<\delta$ for some $A,B\in C(X)$. This implies, $B(a,\delta)\cap B\neq \emptyset$ for all $a\in A$. And therefore, $d'(f(a),f(y))<\epsilon$, $\forall y\in B(a,\delta)\cap B, ~\forall a\in A $; which implies $\inf\limits_{y\in B(a,\delta)\cap B} d'(f(a),f(y))<\epsilon, ~\forall a\in A$. And hence, $\sup\limits_{a\in A}\inf\limits_{y\in B} d'(f(a),f(y))\leq\epsilon$. Similarly, we can prove, $\sup\limits_{b\in B}\inf\limits_{y\in A} d'(f(b),f(y))\leq\epsilon$. Thus we proved, for each $\epsilon>0$ there is a $\delta>0$ such that $H'(F(A),F(B))\leq\epsilon$ whenever $H(A,B)<\delta$, $\forall A,B\in C(X)$.
		
		By continuity and surjectivity of $f$, for each $ P \in C(Y)$ there is a set $S:=\{x\in X: f(x)\in P\}$ which is nonempty and closed in $X$.
		 Thus, $F^{-1}$ exists. Using the uniform continuity of $f^{-1}$, we can show, as above, that $F^{-1}$ is uniformly continuous.
		
	 Conversely, from the uniform continuities of $F$ and $F^{-1}$, it follows that $f$ and $f^{-1}$ are uniformly continuous.
	\end{proof}

We note that, if $f$ is a homeomorphism, then the induced map $F$ need not be a homeomorphism. For instance,
\begin{example}
	Consider $X=(-1,1)$ and $Y=\mathbb R$, both endowed with the usual metric of $\mathbb R$. The map $f:(-1,1) \to \mathbb R$, defined by $f(x)=x/(1-|x|)$, is a homeomorphism, but the induced map $F:C(X)\to C(Y)$ is not. Indeed, the sequence of closed intervals $\{[-n/(n+1),n/(n+1)]\}$ is convergent to $(-1,1)$ in $C(X)$, while the sequence $\{F([-n/(n+1),n/(n+1)])\}$ is not convergent to $F((-1,1))$.
\end{example} 

The above theorem also deduces a necessary and sufficient condition for two Hausdorff metrics to be uniformly equivalent.
\begin{cor}\label{d,d'-H,H'-unif. equi.}
	Let $(X,d)$ be a metric space, and $d'$ be another compatible metric. Then, $d$ and $d'$ are uniformly equivalent if and only if the Hausdorff distances $H$ and $H'$ are uniformly equivalent on $C(X)$.
\end{cor}

A similar result was mentioned by Beer (see \cite{gb93}, Theorem 3.3.2).

\section{Atsuji Hyperspaces}\label{AtsjHprspcs}	
In this section we explore the conditions under which a subspace of a hyperspace $C(X)$ is an Atsuji space. 

The following lemma is a useful tool to derive some of our results in this section.
	


\begin{lemma}\label{UnifHomPrsrvsAtsj}
	Let two metric spaces be uniformly homeomorphic. If one of the spaces is Atsuji, then other is so.
\end{lemma}
\begin{proof}
	Let $f:(X,d)\to (Y,d')$ be a uniform homeomorphism, and $X$ be an Atsuji space. Clearly, $Y'$ is compact in $Y$. Denote the subset $N_\epsilon (Y')\subset Y$ by $S$. If possible, suppose $S^c$ in $Y$ is not uniformly discrete. Then, there are $x_n,y_n$ in $S^c$ such that $d'(x_n,y_n)\to 0$. Since $f^{-1}(S)$ is open in $X$, and contains the compact set $X'$, so there is $\epsilon_1>0$ such that $N_{\epsilon_1}(X')\subset f^{-1}(S)$. By the uniform continuity of $f^{-1}$, $d'(x_n,y_n)\to 0$ implies $d(f^{-1}(x_n),f^{-1}(y_n))\to 0$, and therefore $[N_{\epsilon_1} (X')]^c$ in $X$ is not uniformly discrete, a contradiction.
\end{proof}
The above lemma also implies that uniformly equivalent metrics on a set $X$ generate the same Atsuji subspaces of $X$.
%
%

It is to be noted that Atsujiness is not preserved by homeomorphisms. For instance:

	\begin{example}\label{HomDzntPrsrvAtsjn}
		Consider the set of natural numbers $\mathbb N$ and the set $M:= \{1/n:n\in \mathbb N\}$, both endowed with the usual metric $d$ of $\mathbb R$.
	The map $f:(\mathbb N,d) \to (M,d)$, defined as $f(n)=1/n$, is a homeomorphism. The space $\mathbb N$ is an Atsuji space while $M$ is not.
	\end{example}

A subset $\mathcal C$ in a hyperspace $C(X)$ is said to be \textit{point-finite} if each point $x\in X$ belongs to at most finite number of elements of $\mathcal C$. It is known that a star-finite collection and a locally finite collection of subsets both are point-finite.

We denote a point-finite subset of $C(X)$ by $C_f(X)$, a point-finite subset of $C(X)$ containing all the singletons $\{x\}, x\in X$, by $C_{fs}(X)$, and a subset of $C(X)$ containing all the singletons $\{x\}, x\in X$, by $C^s(X)$.

Given a subset $\mathcal S$ in a hyperspace $C(X)$, let us denote the $\epsilon$-neighborhood of $\mathcal S$, $\bigcup\limits_{A\in \mathcal S}\{B\in C(X):H(A,B)<\epsilon\}$, by $\mathcal N_{\epsilon}(\mathcal S)$. Then, the space $C(X)$ is an Atsuji space if the set of limit points $[C(X)]'$ is compact, and for each $\epsilon>0$, $\big[\mathcal N_{\epsilon}([C(X)]')\big]^c$ is uniformly discrete.

For metric spaces $X$ and their hyperspaces $C(X)$, Gerald Beer \cite{gb85} proved that the following are equivalent: 
\begin{enumerate}
	\item $X$ is either compact or uniformly discrete;
	\item The set $[C(X)]'$ is compact, and for each $\epsilon>0$, $\big[\mathcal N_\epsilon([C(X)]')\big]^c$ is uniformly discrete.
\end{enumerate}

The class of Atsuji spaces contains compact metric spaces as well as uniformly discrete spaces. Replacing  compactness and uniform discreteness of $X$ in Beer's result with the Atsujiness of $X$, the set of limit points $[C(X)]'$ of the hyperspace $C(X)$ may fail to be compact. This is evident from the following example.
\begin{example}\label{C(X)'-noncmpct}
	The subset $P=\{e_m/n:m,n\in \mathbb N\}\cup \{0\}$ of the normed space $(l_2, \|\cdot\|_2)$ is an Atsuji subspace of $X$. The sequence $\{\{0,e_n\}\}_{n=1}^\infty$ is in $[C(P)]'$ with no convergent subsequence.
\end{example}
\noindent However, the set $\big[\mathcal N_\epsilon([C(X)]')\big]^c$ in $C(X)$ remains uniformly discrete for all $\epsilon>0$, as follows:
\begin{thm}\label{[N(C(X)')]^c-N(X')]^c}
	Let $X$ be an Atsuji space. Then, for each space $C^s(X)$ and $\epsilon>0$, the set $\big[\mathcal N_\epsilon([C^s(X)]')\big]^c$ is uniformly discrete.
\end{thm}
\begin{proof}
		If possible, suppose for some $\epsilon>0$, $\big[\mathcal N_\epsilon([C^s(X)]')\big]^c$ in $C^s(X)$ is not uniformly discrete. So, for each $n\in \mathbb N$, there are $P_n,Q_n\in \big[\mathcal N_\epsilon([C^s(X)]')\big]^c$ such that $H(P_n,Q_n)< 1/n$. Since $P_n\in \big[\mathcal N_\epsilon([C^s(X)]')\big]^c$, so $H(P_n,\{l\})\geq \epsilon>\epsilon-\delta, ~\forall l\in X'$, and for some $\delta$ with $\epsilon>\delta>0$. This implies, there exists a sequence $\{p_n\}$ with $p_n\in P_n\cap [N_{\epsilon'}(X')]^c$, where $\epsilon'=\epsilon-\delta$. Since $H(P_n,Q_n)< 1/n$, so for the sequence $\{p_n\}$, there is a sequence $\{q_n\}$ with $q_n\in Q_n$ such that $d(p_n,q_n)< 1/n\to 0$. For the sequence $\{q_{n}\}$, two cases arise: infinitely many points of $\{q_{n}\}$, say $\{q_{n_{j}}\}_{j=1}^{\infty}$, are either in $X'$ or in $[X']^c$. If $\{q_{n_{j}}\}_{j=1}^{\infty}$ is in $X'$, then by the compactness of $X'$, $\{q_{n_{j}}\}$ has a limit point in $X'$. So, the sequence $\{p_n\}$ will have the same limit point as well, in $[N_{\epsilon'}(X')]^c$; and hence $[N_{\epsilon'}(X')]^c$ is not uniformly discrete, which is a contradiction. If $\{q_{n_{j}}\}_{j=1}^{\infty}$ is in $[X']^c$ and it has no limit point, then by the compactness of $X'$ there is a $\delta'> 0$ such that $\{q_{n_{j}}\}_{j=1}^{\infty}\subset [N_{\delta'}(X')]^c$. Indeed, if $\{q_{n_{j}}\}_{j=1}^{\infty}\not\subset [N_{\delta'}(X')]^c$, then for each $k\in \mathbb N$, there is $x'_k\in X'$ such that $B(x'_k,1/k)\cap \{q_{n_{j}}\}_{j=1}^{\infty}\neq \emptyset$, and since $\{x'_k\}$ has a convergent subsequence, so $\{q_{n_{j}}\}_{j=1}^{\infty}$ will have a limit point. Hence, choosing $\epsilon_1=\min\{\epsilon',\delta'\}$, we get the subsequences $\{q_{n_{j}}\}_{j=1}^{\infty}$ and $\{p_{n_{j}}\}_{j=1}^{\infty}$ are in $[N_{\epsilon_1}(X')]^c$, and so $[N_{\epsilon_1}(X')]^c$ is not uniformly discrete, a contradiction.
	\end{proof}
Although, in case of Atsuji space $X$, there is a class of subsets of $C(X)$, such that the set of limit points of each subset from this class is compact.

\begin{thm}\label{X'-[CfX]'-cmplt}
	If $X$ is a metric space for which $X'$ is complete, then for each subset $C_f(X)\subset C(X)$, the set $[C_f(X)]'$ is complete.
\end{thm}

\begin{proof}
	 For a $C\in \big[C_f(X)\big]'$, there is a sequence of distinct terms $\{C_n\}$ in $C_f(X)$ satisfying: for each $n \in \mathbb N, ~ \exists~ Z_n\in \mathbb N$ such that $H(C_r,C)$ $<1/n~ \forall r\geq Z_n.$ This implies, $C\subset N_{1/n}(C_{Z_n})$ for all $n\in \mathbb N$, and so for all $c\in C$, for all $n\in \mathbb N,~ \exists~ c_n \in C_{Z_n}$ such that $d(c_n,c)<1/n$. Since the collection $C_f(X)$ is point-finite, so at most finitely many $c_n$'s can be equal. Thus $c\in X'$, and so $C\in C(X')$. Hence $\big[C_f(X)\big]'\subset C(X')$. We know that, for a complete metric space $W$, the Hausdorff hyperspace $C(W)$ is complete (\cite{gb93}, Theorem 3.2.4). Since $X'$ is complete, the set $C(X')$ is complete in $C(X)$. And, because the set  $[C_f(X)]'$ is a closed subset in $C(X)$, therefore $\big[C_f(X)\big]'$ is complete.
\end{proof}
A similar proof gives the following.
\begin{thm}\label{X'-[CfX]'-cmpct}
	If $X$ is a metric space for which $X'$ is compact, then for each subset $C_f(X)\subset C(X)$, the set $[C_f(X)]'$ is compact.
\end{thm}
A subset $A$ in a metric space $X$ is said to be \textit{totally bounded} if for each $\epsilon>0 $, there are finitely many points $x_1,x_2,...,x_n$ in $X$ such that $A\subset \bigcup\limits_{i=1}^n B(x_i,\epsilon)$. 

The following lemma is immediate.
\begin{lemma}\label{TtlyBnddLemma}
	A subset $A$ in a metric space $X$ is totally bounded if and only if for each $\epsilon>0 $, there are finitely many points $a_1,a_2,...,a_n$ in $A$ such that $A\subset \bigcup\limits_{i=1}^n B(a_i,\epsilon)$.
\end{lemma}

\begin{thm}\label{X-C_{fs}(X)}
	   If $X$ is an Atsuji space, then a completion of each subspace $C_{fs}(X)$ of $C(X)$ is Atsuji.
\end{thm}

\begin{proof}
	Let $\overline{C}_{fs}(X)$ denote the closure of $C_{fs}(X)$ in $C(X)$.  Since $C(X)$ is complete, the space $\overline{C}_{fs}(X)$ is a completion of ${C}_{fs}(X)$. 
	By the proof of Theorem \ref{X'-[CfX]'-cmplt}, we have $[C_{fs}(X)]'\subset C(X')$. Since $X'$ is totally bounded, $C(X')$ is totally bounded (\cite{gb93}, Theorem 3.2.4). Hence, by Lemma \ref{TtlyBnddLemma}, the set $[C_{fs}(X)]'$ is totally bounded in the space $\overline C_{fs}(X)$, which implies $[\overline{C}_{fs}(X)]'$ is compact. And, by Theorem \ref{[N(C(X)')]^c-N(X')]^c}, the set $\overline{C}_{fs}(X)\setminus \mathcal N_\epsilon([\overline C_{fs}(X)]')$ is uniformly discrete. Thus we proved, $\overline{C}_{fs}(X)$ is an Atsuji space. Since any two completions of a metric space are isometric, so by Lemma \ref{UnifHomPrsrvsAtsj}, each completion of $C_{fs}(X)$ is Atsuji.
\end{proof}

\begin{thm}
	If some $C_{fs}(X)$ is an Atsuji space, then $X$ is Atsuji.
\end{thm}
\begin{proof}
	Since the space $C_{fs}(X)$ is Atsuji, and the set $S(X):=\{\{x\}:x\in X\}$ is a closed subset in $C_{fs}(X)$, so $S(X)$ is an Atsuji subspace of $C_{fs}(X)$. Because the space $X$ is uniformly homeomorphic to $S(X)$, using Lemma \ref{UnifHomPrsrvsAtsj}, $X$ is an Atsuji space.
\end{proof}

\begin{rem}
	\textit{For a metric space $X$ with compact $X'$, the collection $A(X)$ of nonempty Atsuji subsets of $X$ need not be a point-finite collection.} For, if the set $X'$ is compact and the collection $A(X)$ is point-finite, then by Corollary \ref{X'-[CfX]'-cmpct}, $[A(X)]'$ is compact. Since $A(X)$ is dense in $C(X)$ (see \cite{gb85}, p. 657), so $[C(X)]'$ is compact, which is not true in general by Example \ref{C(X)'-noncmpct}.
\end{rem}
 Theorem \ref{X-C_{fs}(X)} gives rise to a question: What are the maximal Atsuji subspaces in $C(X)$, provided $X$ is Atsuji? We elaborate the question as follows:
 
\textbf{Open Problem:}
	For a given Atsuji space $X$, consider the collection $\mathcal F$ of the closures of all subsets $C_{fs}(X)$ of $C(X)$. We endow $\mathcal F$ with a partial order relation `$\leq$' as follows: For $\mathcal A, \mathcal B\in \mathcal F$, $\mathcal A\leq \mathcal B$ if and only if $\mathcal A\subset \mathcal B$. Does the partially ordered set $(\mathcal F,\leq)$ have a maximal element? If yes, can one explicitly find the maximal element? 

\begin{thm}
	Let $(X,d)$ be an Atsuji space. If $d'$ is another compatible Atsuji metric on $X$, then the hyperspaces $(C(X),H)$ and $(C(X),H')$ have the same Atsuji subsets.
\end{thm}

\begin{proof}
	By Theorem 2.2 in \cite{gb87}, we get $\tau_H=\tau_{H'}$, where $\tau_H$ is the topology generated by $H$. This implies, the metrics $d, d'$ are uniformly equivalent. Then using Corollary \ref{d,d'-H,H'-unif. equi.}, we get the metrics $H, H'$ are uniformly equivalent. And, so by Lemma \ref{UnifHomPrsrvsAtsj}, the hyperspaces $(C(X),H)$ and $(C(X),H')$ have the same Atsuji subsets. 
\end{proof}


\section{Fixed Point Results}\label{FxdPntsRslts}
Here, we discuss the fixed point results for continuous mappings of Atsuji spaces. 

For a metric space $X$, let the set $P(X)$ be endowed with the Hausdorff distance $H$. Then, $(P(X),H)$ is an extended-real valued pseudo metric space.

\begin{lemma}\label{d(x_n,A_n) to d(x,A)}
	Let $\{x_n\}$ be a sequence converging to $x$ in a metric space $X$. 
	If a sequence $\{A_n\}$ converges to $A$ in $P(X)$, then the sequence $\{d(x_n,A_n)\}$ converges to $d(x,A)$.
\end{lemma}
\begin{proof}
	The proof follows from the continuity of the functional $d(\cdot,K_1):X\to \mathbb R$, and the inequality $d(x,K_1)\leq d(x,K_2)+H(K_1,K_2)$, where $K_1,K_2\in P(X)$.
\end{proof}

Given a multivalued map $f$ from a metric space $X$ to $P(X)$, a point $x\in X$ is said to be \textit{almost fixed point} of $f$, if $\inf\{d(x,y):y\in f(x)\}=0$.

Let $X$ be a complete metric space. Consider that $f:X\to P(X)$ is a continuous map such that for each $\epsilon>0$, there is an $x\in X$ satisfying $d(x,f(x))<\epsilon$; then $f$ does not have an almost fixed point, in general.
Although, in case of Atsuji space $X$, we have the following.
\begin{thm}\label{AtsjAlmstFxdPnt}
	Let $X$ be an Atsuji space, and $f:X\to P(X)$ be a continuous map such that for each $\epsilon>0$, there is an $x\in X$ satisfying $d(x,f(x))<\epsilon$. Then $f$ has an almost fixed point.
\end{thm}
\begin{proof}
	Given hypothesis implies, for each $n\in \mathbb N$, there is $x_n\in X$ such that $d(x_n,f(x_n))<1/n$. If for some $n_0$, $d(x_{n_0}, f(x_{n_0}))=0$, then $f$ has an almost fixed point. Otherwise, for each $n\in \mathbb N$, there is $y_n\in f(x_n)$ such that $0<d(x_n,y_n)<1/n+1/n$. Since $X$ is an Atsuji space, the sequence $\{x_n\}$ has a convergent subsequence $\{x_{n_i}\}_{i=1}^{\infty}$ converging to some $x$ in $X$. Then, by continuity of $f$, $f(x_{n_i})\to f(x)$. Using Lemma \ref{d(x_n,A_n) to d(x,A)}, we have $d(x_{n_i},f(x_{n_i}))\to d(x,f(x))$. This implies $d(x,f(x))=0.$
\end{proof}

In \cite{gb86}, Gerald Beer proved that:  If $X$ is an Atsuji space and $f:X\to X$ is a continuous map such that for some $x\in X$, $\liminf\limits_{n\to \infty}d(f^n(x),f^{n+1}(x))=0$, then $f$ has a fixed point. The following corollary to Theorem \ref{AtsjAlmstFxdPnt} provides a generalization of his result.

\begin{cor}\label{FPfFromAtsjX2CX}
	Let $X$ be an Atsuji space, and $f:X\to C(X)$ be a continuous map such that for each $\epsilon>0$, there is an $x\in X$ satisfying $d(x,f(x))<\epsilon$. Then $f$ has a fixed point. 
\end{cor}

\begin{thm}
		Let $X$ be an Atsuji space; and for a $C_{fs}(X)\subset C(X)$, $f:X\to \overline C_{fs}(X)$ be a map such that for each $\epsilon>0$, there is an $x\in X$ satisfying $H(\{x\},f(x))<\epsilon$. If $f^{-1}$ exists and is continuous, then $f$ has a fixed point. 
\end{thm}
\begin{proof}
	Since $X$ is an Atsuji space, so by Theorem \ref{X-C_{fs}(X)}, $\overline C_{fs}(X)$ is an Atsuji space. By the given hypothesis we have, for each $n\in \mathbb N$, there is an $x_n\in X$ such that $H(\{x_n\},f(x_n))<1/n$. If for some $n_0\in \mathbb N$, $H(\{x_{n_0}\},f(x_{n_0}))=0$, then $x_{n_0}$ is a fixed point for $f$. Otherwise, due to Atsujiness of $\overline C_{fs}(X)$, the sequence $\{f(x_n)\}$ has some convergent subsequence $ \{f(x_{n_i})\}_{i=1}^\infty$ converging to some $A$ in $\overline C_{fs}(X)$. Then, by continuity of $f^{-1}$, the sequence $\{x_{n_i}\}_{i=1}^\infty \subset X$ is convergent to $x:=f^{-1}(A)$. Using Lemma \ref{d(x_n,A_n) to d(x,A)}, we have $d(x,A)=0$, which implies $x\in A=f(x)$.
\end{proof}

For a given subset $X$ in a normed space, we denote the hyperspace of all nonempty closed convex subsets of $X$ by $C_c(X)$.

\begin{thm}\label{FPconvxPrsrvngMap}
	Let $X$ be a compact subspace of a normed space. If $f:X\to X$ is a continuous map which maps a convex set to a convex set and for each $\epsilon>0$ there is $A\in C_c(X)$ satisfying $H(A,F(A))<\epsilon$ for $F$ on $C_c(X)$, then $f$ has a fixed point.
\end{thm}
\begin{proof}
	Since $f$ maps convex sets to convex sets, by the proof of Theorem \ref{fHomFhom}, the induced map $F:C_c(X)\to C_c(X)$ is continuous. It is known that, if a sequence $\{A_n\}$ of nonempty closed convex subsets of a normed space is Wijsman convergent to a nonempty closed subset $A$, then $A$ is convex (\cite{gb93}, p. 43). Thus $C_c(X)$ is closed in $C(X)$, and so compact. Then by Corollary \ref{FPfFromAtsjX2CX}, $F$ has a fixed point, say $P$. This implies, the restricted map $f|_P$ is continuous from $P$ to $P$. And hence, by Tychonoff's Theorem $f$ has a fixed point. 
\end{proof}

We note that, the domain of $f$ taken in Theorem \ref{FPconvxPrsrvngMap} is more general than the domain taken in Schauder's fixed point theorem \cite{sa75}: If $A$ is a compact convex subset of a Banach space and $f$ is a continuous map from $A$ into $A$, then $f$ has a fixed point.

\end{document}